\documentclass[10pt]{article}
\usepackage[top=1.5in,bottom=1in,left=1in,right=1in]{geometry}
\usepackage{amsfonts}
\usepackage{amsthm}
\usepackage{amsmath}

\newtheorem{theorem}{Theorem}[section]
\newtheorem{lemma}[theorem]{Lemma}
\newtheorem{corollary}[theorem]{Corollary}

\newtheorem{remark}[theorem]{Remark}
\newtheorem{proposition}[theorem]{Proposition}

\begin{document}

\title{Asymptotic behavior of averaged and firmly nonexpansive mappings in geodesic spaces}
\author{Adriana Nicolae}  
\date{}
\maketitle

\begin{center}
{\footnotesize
Simion Stoilow Institute of Mathematics of the Romanian Academy, P.O. Box 1-764, 014700 Bucharest, Romania
\ \\
Department of Mathematics, Babe\c s-Bolyai University, Kog\u alniceanu 1, 400084, Cluj-Napoca, Romania
\ \\
E-mail: adriana.nicolae@imar.ro
}
\end{center}

\begin{abstract}
We further study averaged and firmly nonexpansive mappings in the setting of geodesic spaces with a main focus on the asymptotic behavior of their Picard iterates. We use methods of proof mining to obtain an explicit quantitative version of a generalization to geodesic spaces of a result on the asymptotic behavior of Picard iterates for firmly nonexpansive mappings proved by Reich and Shafrir. From this result we obtain effective uniform bounds on the asymptotic regularity for firmly nonexpansive mappings. Besides this, we derive effective rates of asymptotic regularity for sequences generated by two algorithms used in the study of the convex feasibility problem in a nonlinear setting. 

\end{abstract}

\section{Introduction}
In Banach spaces, firmly nonexpansive and averaged mappings form special classes of nonexpansive mappings which possess interesting properties that are not common to all nonexpansive mappings. This paper is mainly motivated by a very recent work by  Ariza-Ruiz, Leu\c stean and L\' opez-Acedo \cite{AriLeuLop12} where firmly nonexpansive mappings are defined and studied in suitable classes of geodesic spaces. 

Firmly nonexpansive mappings were introduced by Bruck \cite{Bru73} in the following way: let $C$ be a nonempty closed and convex subset of a real Banach space $X$. A mapping $T : C \to X$ is {\it firmly nonexpansive} if for each $x,y \in C$ and $t \ge 0$, $\|Tx - Ty\| \le \|t(x-y) + (1-t)(Tx-Ty)\|$. An important example of a firmly nonexpansive mapping is the metric projection onto a closed convex subset of a Hilbert space.

Firmly nonexpansive self-mappings have notable properties when defined in Banach spaces. For instance, if a firmly nonexpansive mapping has fixed points, then it is asymptotically regular and, in more particular settings, the Picard iterates converge weakly to a fixed point of the mapping \cite{ReiSha87}. Such results find natural counterparts in the Hilbert ball \cite{ReiSha87, Rei88, KopRei09}. In \cite{AriLeuLop12}, firmly nonexpansive mappings are studied in geodesic spaces with emphasis on fixed point results, on the asymptotic behavior of Picard iterates for such mappings and on applications thereof.

Averaged mappings already implicitly appeared in Krasnoselski's work \cite{Kra55}, while the term ``averaged'' was coined in \cite{BaBrRe79}. Having a nonempty convex subset $C$ of a normed space $X$,  an {\it averaged} mapping $T_\lambda : C \to C$ has the form $T_\lambda = (1-\lambda) I + \lambda T$, where $\lambda \in (0,1)$, $I$ is the identity mapping and $T$ is nonexpansive. The metric projection onto a closed convex subset of a Hilbert space is averaged. In fact, in Hilbert spaces, a mapping is firmly nonexpansive if and only if it is averaged with $\lambda = 1/2$. A well-known result concerning averaged mappings was proved by Ishikawa in \cite{Ish76} and states that in any Banach space, $T_\lambda$ is asymptotically regular whenever it has bounded orbits. Ishikawa's result has been extended by Goebel and Kirk in \cite{GoeKir83} to the setting of hyperbolic spaces. In particular Banach spaces, the Picard iterates of averaged mappings converge weakly to a fixed point of the mapping provided the mapping is not fixed point free. The asymptotic behavior of averaged mappings has been further studied in the context of Banach spaces by Baillon, Bruck and Reich \cite{BaBrRe79} and in the Hilbert ball by Reich \cite{Rei85}. Averaged mappings have remarkable applications. Many iterative algorithms used for instance in signal processing and image reconstruction employ averaged mappings for particular choices of the nonexpansive mapping $T$. In the setting of Hilbert and Banach spaces, such algorithms have been intensively studied in the past decades (for a unified treatment of some of these algorithms see, for instance, \cite{Byr03}). 

Our first main goal is to further study averaged and firmly nonexpansive mappings in geodesic spaces. To this end we apply methods of proof mining to give effective results on the asymptotic behavior of averaged and firmly nonexpansive mappings. Proof mining deals with the analysis of mathematical proofs with the goal of extracting additional information from them. By analyzing proofs using methods from logic, one may be able to strengthen conclusions of results mostly by obtaining effective bounds and uniformities in the bounds. More details about proof mining techniques can be found in \cite{Koh08}. Thus, in Section \ref{sect_betw} we study a betweenness property of metric spaces which is useful for relating periodic and fixed points for firmly nonexpansive and averaged mappings. We provide significant examples of metric spaces that satisfy this property. In Section \ref{sect_effective} we gather results on the asymptotic behavior of averaged mappings and provide a quantitative version of a generalization to geodesic spaces of a very important result due to Reich and Shafrir \cite{ReiSha87} on the asymptotic behavior of Picard iterates of firmly nonexpansive mappings. As an immediate consequence we obtain an exponential rate of asymptotic regularity for the Picard iterates. We remark that in \cite{AriLeuLop12} another rate of asymptotic regularity was established for $UCW$-hyperbolic spaces which, in particular, is quadratic in the case of CAT$(0)$ spaces. However, our exponential rate is the only known rate which holds even in the setting of normed spaces.

The second main goal of the paper is to derive effective rates of asymptotic regularity for sequences obtained by two methods used in the study of the convex feasibility problem in a nonlinear context. These two schemes are different forms of a very general method introduced by Bauschke and Borwein in \cite{BauBor96} to provide a unified treatment of various methods used in the study of the convex feasibility problem. The convex feasibility problem consists in finding a point in the intersection of finitely many sets if such a point exists. The first algorithm that we focus on is the well-known alternating projection method developed by von Neumann \cite{vNeu33} (see also \cite{KopRei04} for an elementary geometric proof). This method was studied in the Hilbert ball by Reich \cite{Rei93} and was recently extended to the setting of CAT$(0)$ spaces by Ba\v c\' ak, Searston and Sims \cite{BacSeaSim12}. In fact, one obtains linear convergence of this method even in a nonlinear setting if additional assumptions are imposed on the sets \cite{BauBor93, BacSeaSim12}. In the general case when no further conditions are assumed on the sets one can only obtain weak convergence of the method \cite{Hun04, MatRei03}. For this case we give here a quadratic rate of asymptotic regularity for the sequence generated by the alternating projection method. Additionally, we focus on a parallel projection scheme (see \cite{Com97} for more references on parallel projection methods) defined in terms of weighted averages of nonexpansive retractions which extends a method studied in Hilbert spaces by Crombez \cite{Cro91} and later generalized by Takahashi and Tamura \cite{TakTam96} in the setting of uniformly convex Banach spaces. We show that this method finds a natural counterpart in the geodesic setting and give a rate of asymptotic regularity for $UCW$-hyperbolic spaces. In the particular setting of CAT$(0)$ spaces, this rate is quadratic and it is polynomial for $L_p$ spaces with $1 < p < \infty$. At the same time we prove the $\Delta$-convergence of this method, showing that it can be used in a nonlinear setting. To our knowledge, the quantitative results that we obtain are new even in the Hilbert setting.

\section{Basic notions on geodesic spaces} \label{sect_prelim}
Let $(X,d)$ be a metric space.  A {\it geodesic path} from $x$ to $y$ is a mapping $c:[0,l] \to X$, where $[0,l] \subseteq \mathbb{R}$, such that $c(0) = x, c(l) = y$ and $d\left(c(t),c(t^{\prime})\right) = \left|t - t^{\prime}\right|$ for every $t,t^{\prime} \in [0,l]$. The image $c\left([0,l]\right)$ of $c$ forms a {\it geodesic segment} which joins $x$ and $y$. Note that a geodesic segment from $x$ to $y$ is not necessarily unique. If no confusion arises, we use  $[x,y]$ to denote a geodesic segment joining $x$ and $y$.  $(X,d)$ is a {\it (uniquely) geodesic space} if every two points $x,y \in X$ can be joined by a (unique) geodesic path. A point $z\in X$ belongs to the geodesic segment $[x,y]$ if and only if there exists $t\in [0,1]$ such that $d(z,x)= td(x,y)$ and $d(z,y)=(1-t)d(x,y)$, and we write $z=(1-t)x+ty$ for simplicity. This, too, may not be unique. A subset $C$ of $X$ is {\it convex} if $C$ contains any geodesic segment that joins every two points in $C$. For a comprehensive treatment of geodesic metric spaces one may see for example \cite{Bri99, Pap05}. 

The {\it metric} $d: X \times X \to \mathbb{R}$ is said to be {\it convex} if for any $x,y,z \in X$ one has
\[d(x,(1-t)y + tz) \le (1-t)d(x,y) + td(x,z) \mbox{ for all } t \in [0,1].\]

A geodesic space $(X,d)$ is {\it Busemann convex} if given any pair of geodesic paths $c_1 : [0, l_1] \to X$ and $c_2 : [0,l_2] \to X$  one has 
\[d(c_1(tl_1),c_2(tl_2)) \le (1-t)d(c_1(0),c_2(0)) + td(c_1(l_1),c_2(l_2)) \mbox{ for all } t \in [0,1].\] 
Any Busemann convex space has convex metric.

$W$-hyperbolic spaces are defined in \cite{Koh08} using the notion of convexity mapping. The triple $(X,d,W)$ is a {\it W-hyperbolic space} if $(X,d)$ is a metric space and $W : X \times X \times [0,1] \to X$ is a convexity mapping such that for each $x,y,z,w \in X$ and for any $t,t' \in [0,1]$ we have
\begin{itemize}
\item[(W1)] $d(z,W(x,y,t)) \le (1-t)d(z,x) + td(z,y)$,
\item[(W2)] $d(W(x,y,t),W(x,y,t')) = |t-t'|d(x,y)$,
\item[(W3)] $W(x,y,t) = W(y,x,1-t)$,
\item[(W4)] $d(W(x,z,t),W(y,w,t)) \le (1-t)d(x,y)+td(z,w)$.
\end{itemize}
Notice that hyperbolic spaces are defined in different ways in the literature (see, for instance, \cite{GoeKir83, GoeRei84, ReiSha90}).

Any $W$-hyperbolic space is a geodesic space and it is uniquely geodesic if and only if for every $x,y \in X$, $x \ne y$ and each $\lambda \in (0,1)$ there exists a unique $z \in X$ (namely $z = W(x,y,\lambda)$) such that $d(x,z) = \lambda d(x,y)$ and $d(y,z) = (1-\lambda) d(x,y)$. Note also that a metric space is Busemann convex if and only if it is a uniquely geodesic $W$-hyperbolic space (see \cite{AriLeuLop12} for details).

A {\it geodesic triangle} $\Delta(x_1,x_2,x_3)$ consists of three points $x_1, x_2$ and $x_3$ in $X$ (its {\it vertices}) and three geodesic segments corresponding to each pair of points (its {\it edges}). For the geodesic triangle $\Delta$=$\Delta(x_1,x_2,x_3)$, a {\it comparison triangle} is a triangle $\bar{\Delta} = \Delta(\bar{x}_1, \bar{x}_2, \bar{x}_3)$ in $\mathbb{R}^2$ such that $d(x_i,x_j) = d_{\mathbb{R}^2}(\bar{x}_i,\bar{x}_j)$ for $i,j \in \{1,2,3\}$. Comparison triangles of geodesic triangles always exist and are unique up to isometry.

A geodesic triangle $\Delta$ satisfies the {\it CAT$(0)$ inequality} if for every comparison triangle $\bar{\Delta}$ of $\Delta$ and for every $x,y \in \Delta$ we have
\[d(x,y) \le d_{\mathbb{R}^2}(\bar{x},\bar{y}),\]
where  $\bar{x},\bar{y} \in \bar{\Delta}$ are the corresponding points of $x$ and $y$, i.e., if $x = (1-t)x_i + tx_j$ then $\bar{x} = (1-t)\bar{x}_i + t\bar{x}_j$.

A {\it CAT$(0)$ space} (also known as a space of bounded curvature in the sense of Gromov) is a geodesic space for which every geodesic triangle satisfies the CAT$(0)$ inequality. CAT$(0)$ spaces have attracted the attention of a large number of researchers in the last few decades due to their rich geometry and relevance in different problems. The fact that a CAT$(0)$ space is Busemann convex has a great impact on the geometry of the space, but being Busemann convex is a weaker property than being CAT$(0)$. The following inequality is called the {\it (CN) inequality} and is equivalent to the CAT$(0)$ condition. Let $x,y_1,y_2$ be points in a CAT$(0)$ space and let $m = (1-t)y_1 +ty_2$ for some $t \in [0,1]$. Then,
\[d\left(x,m\right)^2 \le (1-t)d(x,y_1)^2 + td(x,y_2)^2 - t(1-t)d(y_1,y_2)^2.\] 

A geodesic space $(X,d)$ is {\it uniformly convex} \cite{GoeRei84, ReiSha90}  if for any $r > 0$ and $\varepsilon \in (0,2]$ there exists $\delta \in (0,1]$ such that if $a,x,y \in X$ with $d(x,a) \le r$, $d(y,a) \le r$ and $d(x,y) \ge \varepsilon r$ then
\[d\left(\frac{1}{2}x + \frac{1}{2}y,a\right) \le (1 - \delta)r.\]
A mapping $\delta : (0,\infty) \times (0,2] \to (0,1]$ providing such a $\delta = \delta(r,\varepsilon)$ for a given $r>0$ and $\varepsilon \in (0,2]$ is called a {\it modulus of uniform convexity}. The mapping $\delta$ is {\it monotone} if for every fixed $\varepsilon$ it decreases with respect to $r$. CAT$(0)$ spaces are uniformly convex metric spaces admitting a modulus of uniform convexity which does not depend on the radius of the balls.

In fact, uniform convexity was defined in the setting of $W$-hyperbolic spaces in \cite{Leu07} similarly to above. Uniformly convex $W$-hyperbolic spaces with a monotone modulus of uniform convexity were called {\it $UCW$-hyperbolic spaces}. However, any uniformly convex $W$-hyperbolic space is Busemann convex so $UCW$-hyperbolic spaces form a particular class of uniformly convex metric spaces with a monotone modulus of uniform convexity.

If we drop in the above definition the uniformity conditions then we find the notion of {\it strict convexity} in metric spaces. Consequently, every uniformly convex geodesic space is strictly convex. Moreover, any Busemann convex space is strictly convex. Note that strictly convex metric spaces are uniquely geodesic.


\section{Relating periodic points to fixed points through a betweenness property} \label{sect_betw}
In this section we focus on the relation between periodic and fixed points for averaged and firmly nonexpansive mappings. We recall first a betweenness property which plays an essential role in proving fixed point results for these two classes of mappings in geodesic spaces. We say that a point $y$ in a metric space {\it lies between} two distinct points $x$ and $z$ if $y$ is distinct from $x$ and $z$ and $d(x, z) = d(x, y) + d(y, z)$. In geodesic spaces, the fact that $y$ lies between $x$ and $z$ is equivalent to the fact that the three points are pairwise distinct and $y$ belongs to a geodesic segment joining $x$ and $z$. 
	
In any metric space $(X,d)$ we have the following transitivity property for betweenness which we call in this work the {\it weak betweenness property} (see \cite[Proposition 2.2.13]{Pap05}): if $x,y,z,w$ are pairwise distinct points in $X$, then $y$ lies between $x$ and $z$ and $z$ lies between $x$ and $w$ if and only if $y$ lies between $x$ and $w$ and $z$ lies between $y$ and $w$.

We say that $X$ satisfies the {\it betweenness property} if for every four pairwise distinct points $x,y,z,w \in X$, if $y$ lies between $x$ and $z$ and $z$ lies between $y$ and $w$, then $y$ and $z$ lie between $x$ and $w$. Note that the betweenness property induces the following property: for $n \ge 2$ and $x_0, x_1, \ldots, x_n \in X$, if $x_k$ lies between $x_{k-1}$ and $x_{k+1}$ for each $1 \le k \le n-1$, then $x_k$ lies between $x_0$ and $x_{k+1}$ for each $1 \le k \le n-1$. 

Let $(X,d)$ be a geodesic space, $C \subseteq X$ nonempty and convex, $\lambda \in (0,1)$ and $T : C \to C$ nonexpansive. The mapping $T_\lambda : C \to C$ defined by $T_\lambda x = (1 - \lambda)x + \lambda Tx$ is called {\it averaged} . Note that $x$ and $Tx$ must not be joined by a unique geodesic segment, $T_\lambda x$ being assumed to belong to one such fixed segment (more precisely one can consider a convexity mapping which assigns to each two points exactly one geodesic segment joining them). 

It is immediate that $T_\lambda$ and $T$ have the same fixed point set. Also, if $X$ is Busemann convex, then any averaged mapping is nonexpansive. In general, we have that $d(T_\lambda ^2x,T_\lambda x) \le d(T_\lambda x, x)$ for each $x \in C$. Indeed,
\begin{align*}
d(T_\lambda^2 x, T_\lambda x) & = d\left((1-\lambda)T_\lambda x + \lambda T(T_\lambda x), T_\lambda x\right) \\
& =  \lambda d(T_\lambda x, T(T_\lambda x)) \le \lambda \left(d(T_\lambda x, Tx) +  d(Tx, T(T_\lambda x))\right)\\
& \le \lambda \left(d(T_\lambda x, Tx) + d(x, T_\lambda x) \right) = \lambda d(x,Tx) = d(T_\lambda x, x).
\end{align*}

Let $X$ be a metric space and $C \subseteq X$. Recall that a mapping $T : C \to C$ is said to be {\it asymptotically regular at $x \in C$} if $\displaystyle \lim_{n \to \infty}d(T^n x, T^{n+1}x) = 0$ and it is {\it asymptotically regular} if it is asymptotically regular at each $x \in C$. This concept was introduced by Browder and Petryshyn for normed spaces in \cite{BroPet66}. More generally, we say that a sequence $(x_n) \subseteq X$ is {\it asymptotically regular} if $\displaystyle \lim_{n \to \infty} d(x_n, x_{n+1}) =0$. A rate of convergence of $(d(x_n, x_{n+1}))$ towards $0$ will be called a {\it rate of asymptotic regularity}.

In the case of averaged mappings, it follows from \cite[Proposition 2]{GoeKir83} proved by Goebel and Kirk that in any geodesic space with convex metric, if $(T_\lambda^n x)$ is bounded, then $T_\lambda$ is asymptotically regular at $x \in C$. Thus, in such a context any periodic point of an averaged mapping is a fixed point. We focus next on the relation between periodic and fixed points for averaged mappings in geodesic spaces with the betweenness property.

\begin{proposition}\label{prop_betw-averaged}
Let $X$ be a geodesic space with the betweenness property, $C$ a nonempty and convex subset of $X$. Suppose $T_\lambda : C \to C$ is an averaged mapping. Then any periodic point of $T_\lambda$ is a fixed point of $T_\lambda$.
\end{proposition}
\begin{proof}
Suppose $x$ is a periodic point which is not fixed and take $m \ge 1$ minimal such that $x = T_\lambda^{m+1}x$. Then, 
\begin{align*}
d(x,T_\lambda^m x) &= d(T_\lambda^{m+1}x,T_\lambda^m x) \le d(T_\lambda^m x, T_\lambda^{m-1} x) \le \ldots \le d(T_\lambda x,x) \\
&= d(T_\lambda^{m+2}x, T_\lambda^{m+1} x) \le d(T_\lambda^{m+1} x, T_\lambda^m x) = d(x, T_\lambda^m x),
\end{align*}
which yields
\[d(T_\lambda x,x) = d(T_\lambda ^2x,T_\lambda x) = \ldots = d(T_\lambda ^mx, T_\lambda ^{m-1}x) = d(x,T_\lambda ^m x) = \gamma > 0.\]
Since $T_\lambda$ is averaged, for every $1 \le k \le m$ we have
\begin{align*}
\gamma & = d(T_\lambda ^{k+1} x, T_\lambda ^k x) = \lambda d\left(T_\lambda ^k x, T(T_\lambda ^k x)\right) \le \lambda \left( d\left(T_\lambda^k x, T(T_\lambda^{k-1} x)\right) + d\left(T(T_\lambda^{k-1} x), T(T_\lambda ^k x)\right) \right) \\
&\le \lambda \left( d\left(T_\lambda^k x, T(T_\lambda^{k-1} x)\right) + d\left(T_\lambda^{k-1} x, T_\lambda ^k x\right) \right) = \lambda d\left(T_\lambda^{k-1} x, T(T_\lambda^{k-1} x)\right) = d(T_\lambda^k x, T_\lambda ^{k-1} x) = \gamma.
\end{align*} 
Thus, for each $1 \le k \le m$, $d\left(T_\lambda^k x, T(T_\lambda ^k x)\right) = d\left(T_\lambda^k x, T(T_\lambda^{k-1} x)\right) + d\left(T(T_\lambda^{k-1} x), T(T_\lambda ^k x)\right)$. If for some $k$, $1 \le k \le m$, $T(T_\lambda^{k-1} x) = T_\lambda^k x$ or $T(T_\lambda^{k-1} x) = T(T_\lambda^k x)$, then we obtain that $\gamma = 0$, a contradiction. Indeed, if $T(T_\lambda^{k-1} x) = T_\lambda^k x$, then $T_\lambda^{k-1} x = T_\lambda^k x$, so $\gamma = 0$. If $T(T_\lambda^{k-1} x) = T(T_\lambda^k x)$, then
\[\frac{\gamma}{\lambda} = d\left(T_\lambda^k x, T(T_\lambda^k x)\right) = d\left(T_\lambda^k x, T(T_\lambda^{k-1} x)\right) = (1-\lambda) d\left(T_\lambda^{k-1} x, T(T_\lambda^{k-1} x)\right) = \frac{1-\lambda}{\lambda} \gamma,\]
which implies that $\gamma = 0$. Therefore, for each $1 \le k \le m$, $T(T_\lambda^{k-1} x)$ lies between $T_\lambda^k x$ and $T(T_\lambda ^k x)$.\\
If $m = 1$ then $k=1$ so $T(T_\lambda^{k-1} x) = Tx$ lies between $T(T_\lambda x)$ and $T_\lambda x$. Since $T_\lambda x$ lies between $x$ and $Tx$ we have by the betweenness property that $T_\lambda x$ lies between $T(T_\lambda x)$ and $x$. Thus,
\[d\left(x, T(T_\lambda x)\right) = d(x, T_\lambda x) + d\left(T_\lambda x, T(T_\lambda x)\right). \]
Since $T_\lambda ^2 x = x$ this yields that
\[d\left(x, T(T_\lambda x)\right) = d(x, T_\lambda x) + d(x, T_\lambda x) + d\left(x, T(T_\lambda x)\right),\]
which is impossible.\\
Suppose now $m \ge 2$. Since $T_\lambda^k x$ lies between $T_\lambda^{k-1} x$ and $T(T_\lambda^{k-1} x)$ and $T(T_\lambda^{k-1} x)$ lies between $T_\lambda^k x$ and $T(T_\lambda^k x)$ we apply again the betweenness property to get that $T_\lambda^k x$ lies between $T_\lambda^{k-1} x$ and $T(T_\lambda^k x)$. Knowing that $T_\lambda^{k+1} x$ lies between $T_\lambda^k x$ and $T(T_\lambda^k x)$, by applying the weak betweenness property, we obtain that $T_\lambda^k x$ lies between $T_\lambda^{k-1} x$ and $T_\lambda^{k+1} x$ for all $k = 1, \ldots, m$. Finally we obtain by the betweenness property that $T_\lambda^{m-1} x$ lies between $x$ and $T_\lambda^m x$, so
\[\gamma = d(x,T_\lambda^m x)= d(T_\lambda^{m-1} x,T_\lambda^m x)+d(T_\lambda^{m-1} x,x)=\gamma+d(T_\lambda^{m-1} x,x)>\gamma,\]
a contradiction.
\end{proof}

We focus next on the connection between periodic and fixed points for firmly nonexpansive mappings. Let $(X,d)$ be a geodesic space, $C \subseteq X$ and $T : C \to X$. For $\lambda \in (0,1)$, we say that the mapping $T$ is {\it $\lambda$-firmly nonexpansive} if $T$ is nonexpansive and for all $x,y \in C$,
\begin{equation} \label{def_firmly-non}
d(Tx,Ty) \le d((1-\lambda)x + \lambda Tx, (1-\lambda)y + \lambda Ty).
\end{equation} 
Note that in the above definition if the space is not uniquely geodesic one should fix beforehand for each $z \in C$ a geodesic segment joining $z$ and $Tz$.

If $T$ is $\lambda$-firmly nonexpansive for every $\lambda \in (0,1)$, where the choice of segments joining $z$ to $Tz$ for $z \in C$ does not depend on $\lambda$, then $T$ is called {\it firmly nonexpansive}.

Note that in general (\ref{def_firmly-non}) does not necessarily imply nonexpansivity. However, condition (\ref{def_firmly-non}) implies that $d(T^2x,Tx) \le d(Tx, x)$ for each $x \in C$. Indeed,
\begin{align*}
d(T^2x, Tx) & \le d((1-\lambda)Tx + \lambda T^2x, (1-\lambda)x + \lambda Tx) \\
& \le  d((1-\lambda)Tx + \lambda T^2x, Tx) +  d(Tx, (1-\lambda)x + \lambda Tx)\\
& = \lambda d(T^2x, Tx) + (1-\lambda)d(Tx,x).
\end{align*}

If the space is Busemann convex then nonexpansivity is guaranteed by (\ref{def_firmly-non}). Besides, in the context of Busemann convex spaces, we have even more for firmly nonexpansive mappings. Let $x,y \in C$ and $\varphi_{x,y} : [0,1] \to \mathbb{R}_+$ be defined by
\[\varphi_{x,y}(\lambda) = d((1-\lambda)x + \lambda Tx, (1-\lambda)y + \lambda Ty).\]
Then $\varphi_{x,y}$ is a convex function and, since by (\ref{def_firmly-non}) $\varphi_{x,y}(1) \le \varphi_{x,y}(\lambda)$ for every $\lambda \in [0,1]$, we have that $\varphi_{x,y}$ is non-increasing. 

In \cite[Proposition 4.3]{AriLeuLop12} it was proved that in the context of Busemann convex spaces, any periodic point of a $\lambda$-firmly nonexpansive mapping is also fixed. However, this result is immediate since by \cite[Corollary 5.3]{AriLeuLop12}, we know that the mapping is asymptotically regular because the orbit of a periodic point is bounded. Hence any periodic point is fixed. Note that this reasoning can also be applied in geodesic spaces where the metric is convex since this setting is enough for \cite[Corollary 5.3]{AriLeuLop12} to hold. In fact it can be proved that in any geodesic space with the betweenness property, any periodic point of a mapping satisfying (\ref{def_firmly-non}) is fixed. We omit the details of the proof since it can be obtained by following an argument similar to that in \cite{AriLeuLop12}.

\begin{proposition}\label{prop_betw-fixed}
Let $X$ be a geodesic space with the betweenness property, $C \subseteq X$ nonempty and $T : C \to C$ a mapping that satisfies (\ref{def_firmly-non}) for some $\lambda \in (0,1)$. Then any periodic point of $T$ is a fixed point of $T$.
\end{proposition}

\subsection{Some examples of metric spaces with the betweenness property}
A metric space $(X,d)$ is called a {\it Ptolemy metric space} if
\[d(x,z)d(y,w) \le d(x,y)d(z,w) + d(x,w)d(y,z) \mbox{ for every } x,y,z,w \in X.\]
The above relation is known as the Ptolemy inequality. This inequality is significant in both the normed and the metric setting. It is known that a normed space is an inner product space if and only if it is a Ptolemy space \cite{Sch52}. CAT$(0)$ spaces are Ptolemy spaces, but geodesic Ptolemy spaces are not necessarily CAT$(0)$ \cite{Foe07}. The Ptolemy inequality proved to play an essential role in, for instance, the study of the boundary at infinity of CAT$(-1)$ spaces \cite{FoeSch11}. 
\begin{proposition} \label{prop_betw-Ptolemy}
Every Ptolemy metric space satisfies the betweenness property.
\end{proposition}
\begin{proof}
Let $x,y,z,w$ be four pairwise distinct points in a Ptolemy metric space such that $y$ is between $x$ and $z$ and $z$ is between $y$ and $w$. By the Ptolemy inequality we have
\[d(x,z)d(y,w) \le d(x,y)d(z,w) + d(x,w)d(y,z).\]
Thus,
\[\left(d(x,y) + d(y,z)\right) \left(d(y,z) + d(z,w)\right) \le d(x,y)d(z,w) + d(x,w)d(y,z),\]
that is,
\[d(x,y)d(y,z) + d(x,y)d(z,w) + d(y,z)^2 + d(y,z)d(z,w) \le d(x,y)d(z,w) + d(x,w)d(y,z),\]
so,
\[d(x,y) + d(y,z) + d(z,w) \le d(x,w).\]
From this we have on the one hand that $d(x,z) + d(z,w) = d(x,w)$ and on the other that $d(x,y) + d(y,w) = d(x,w)$. This means that $y$ and $z$ both lie between $x$ and $w$.
\end{proof}

The betweenness property holds in Busemann convex spaces (see \cite[Proposition 8.2.4]{Pap05}). Note that geodesic Ptolemy spaces are not necessarily Busemann convex. In fact they are not even uniquely geodesic (see \cite{Foe07}) which shows that in the context of geodesic spaces, the betweenness property does not yield the uniqueness of geodesic segments. However, Busemann convexity implies the CAT$(0)$ condition in the setting of Ptolemy spaces \cite{Foe07}. Other properties of geodesic Ptolemy spaces in connection to metric fixed point theory can be found in \cite{EspNic11, EspNic12}.

We see below that the betweenness property holds in the more general context of geodesic spaces with convex metric.

\begin{proposition}
Every geodesic space with convex metric satisfies the betweenness property.
\end{proposition}
\begin{proof}
Let $x,y,z,w$ be four pairwise distinct points such that $y$ is between $x$ and $z$ and $z$ is between $y$ and $w$. Suppose $z =(1-r)y + rw$, where $r \in (0,1)$. Then, $d(x,z) \le (1-r) d(x,y) + rd(x,w)$, from where
\[r d(x,w) \ge d(x,z) - d(x,y) + rd(x,y) = d(y,z) + r d(x,y) = r d(y,w) + r d(x,y).\]
Hence, $d(x,y) + d(y,w) \le d(x,w)$, which implies that $y$ lies between $x$ and $w$. But then,
\[d(x,w) = d(x,y) + d(y,w) = d(x,y) + d(y,z) + d(z,w) = d(x,z) + d(z,w),\]
which yields that $z$ also lies between $x$ and $w$.
\end{proof}

Let $\kappa \in \mathbb{R}$. CAT$(\kappa)$ spaces are defined in terms of comparisons with the model spaces $M_\kappa^2$ similarly to CAT$(0)$ spaces. We denote the {\it diameter of $M_\kappa ^2$} by $D_\kappa$. More precisely, for $\kappa > 0$, $D_\kappa = \pi/\sqrt{\kappa}$ and for $\kappa \le 0$, $D_\kappa = \infty$. In any CAT$(\kappa)$ space $X$ having $x_0 \in X$, the restriction of $x \mapsto d(x,x_0)$ to the open ball centered at $x_0$ and of radius $D_\kappa/2$ is convex. Therefore, on sets of diameter less than $D_\kappa/2$ the betweenness property is satisfied. For more details about these spaces and related topics one can consult Bridson and Haefliger \cite{Bri99}.  

We recall next some notions needed below. A geodesic path $c : [0,l] \to X$ is {\it extendable} beyond the point $c(l)$ if $c$ is a restriction of a geodesic path $c':[0, l'] \to X$ with $l' > l$. We say that two {\it geodesics bifurcate} if they have a common endpoint and coincide on an interval, but one is not an extension of the other. Alexandrov spaces of curvature bounded below cannot have bifurcating geodesics. We refer the reader to the paper of Burago, Gromov and Perelman \cite{BurGroPer92} and to Plaut \cite{Plaut02} for a detailed discussion on the geometry of Alexandrov spaces.

\begin{proposition}
Let $X$ be a geodesic space without bifurcating geodesics and assume that every geodesic path is extendable. Then $X$ satisfies the betweenness property.
\end{proposition}
\begin{proof}
Let $x,y,z,w \in X$ be pairwise distinct such that $y$ is between $x$ and $z$ and $z$ is between $y$ and $w$. Denote by $[x,z]$ a geodesic segment that joins $x$ and $z$ and contains $y$ and by $[y,w]$ a geodesic segment joining $y$ and $w$ and containing $z$. 
Since $X$ has no bifurcating geodesics it follows that $y$ and $z$ are joined by a unique geodesic segment. Consider $z' \in [z,w] \subseteq [y,w]$ such that $d(x,z) + d(z,z') = d(x,z')$ and $d(z,z') \in [0, d(z,w)]$ is maximal. Suppose $d(z,z') < d(z,w)$. Let $c : [0,l] \to X$ be the geodesic path whose image is $[x,z] \cup [z,z']$ with $c(l) = z'$. Then there exist $\varepsilon > 0$ and a geodesic path $c' : [0, l+ \varepsilon] \to X$ such that $c'\vert_{[0,l]} = c$. Pick $0 < \delta < \min\{\varepsilon, d(z', w)\}$ and $z_\delta \in [z', w]$ for which $d(z',z_\delta) = \delta$. Because the space has no bifurcating geodesics we have that $[y,z_\delta]$ and $c'\left([d(x,y), l+\delta]\right)$ coincide and $z_\delta = c'(l + \delta)$, which contradicts the maximality of $z'$. Hence $z' = w$ and we are done.
\end{proof}

\begin{corollary}\label{cor_betw-Alex}
Every Alexandrov space of curvature bounded below for which all geodesic paths are extendable satisfies the betweenness property.
\end{corollary}

\section{Effective results on the asymptotic behavior of Picard iterates}\label{sect_effective}

\subsection{Averaged mappings}

In this section we gather some facts on averaged mappings in the context of geodesic spaces. We study next the asymptotic behavior of Picard iterates for averaged mappings in geodesic spaces giving analogues of the results obtained in Banach spaces by Baillon, Bruck and Reich \cite{BaBrRe79} and in the Hilbert ball by Reich \cite{Rei85}. 

\begin{lemma}\label{lemma_averaged}
Let $X$ be a geodesic space with convex metric and $C \subseteq X$ nonempty and convex. Suppose that $T_\lambda : C \to C$ is averaged. Then, for every $x,y \in C$,
\[d(T_\lambda x,T_\lambda y) \le (1-\lambda) d(T_\lambda x,y) + \lambda(1-\lambda)d(x,T_\lambda y) + (1-\lambda)^2d(y,T_\lambda y) + \lambda^2d(x,y).\]
\end{lemma}
\begin{proof}
Let $x,y \in C$. Then,
\begin{align*}
d(T_\lambda x,T_\lambda y) &= d(T_\lambda x, (1-\lambda)y + \lambda Ty) \le (1-\lambda)d(T_\lambda x, y) + \lambda d(T_\lambda x, Ty)\\
& = (1-\lambda)d(T_\lambda x, y) + \lambda d((1-\lambda)x + \lambda Tx, Ty)\\ 
& \le (1-\lambda)d(T_\lambda x, y) + \lambda(1-\lambda) d(x, Ty) + \lambda^2 d(Tx,Ty)\\
& \le (1-\lambda)d(T_\lambda x, y) + \lambda(1-\lambda) d(x, T_\lambda y) + \lambda(1-\lambda)d(T_\lambda y, Ty) + \lambda^2 d(x,y)\\
& = (1-\lambda) d(T_\lambda x,y) + \lambda(1-\lambda)d(x,T_\lambda y) + (1-\lambda)^2d(y,T_\lambda y) + \lambda^2d(x,y).
\end{align*}
\end{proof}

We already know by \cite[Proposition 2]{GoeKir83} that if $(T_\lambda^n x)$ is bounded then it is asymptotically regular at $x \in C$. Moreover, if $T_\lambda$ is also nonexpansive, then all orbits will be bounded and so $T_\lambda$ will be asymptotically regular. Additionally, we have the following.
\begin{theorem}
Let $X$ be a geodesic space with convex metric and $C \subseteq X$ nonempty and convex. Suppose that $T_\lambda : C \to C$ is averaged and nonexpansive. Then for every $x \in C$ and $k \in \mathbb{N}$,
\[\lim_{n \to \infty}d(T_\lambda^{n+1}x,T_\lambda^n x) = \frac{1}{k} \lim_{n \to \infty} d(T_\lambda^{n+k}x, T_\lambda^n x) = \lim_{n \to \infty}\frac{d(T_\lambda^n x, x)}{n} = r_C(T_\lambda),\]
where $r_C(T) =\inf\{d(x,Tx) : x\in C\}$ is the minimal displacement of  $T$.
\end{theorem}
\begin{proof}
The result can be proved by applying Lemma \ref{lemma_averaged} and \cite[Lemma 5.4]{AriLeuLop12} using reasoning similar to that adopted in the case of firmly nonexpansive mappings which was discussed in \cite{AriLeuLop12}. A slightly different proof can be obtained by an immediate adaptation of the proof given in \cite{Rei85} for the Hilbert ball.
\end{proof}

The above result shows that if $T_\lambda$ is asymptotically regular at some point $x \in C$ then $T_\lambda$ is asymptotically regular which in turn is equivalent to $r_C(T_\lambda) = 0$. 

In Banach spaces, conditions for the convergence of the sequence $(T^n x/ n)$ when the mapping $T$ is nonexpansive are discussed in \cite{Paz71, Rei73, Rei80, KohNey81, Rei81, Rei82, PlaRei83}.

We focus next on rates of asymptotic regularity for averaged mappings. Let $C\subseteq X$ be a nonempty and convex subset of a geodesic space and $T : C \to C$ be nonexpansive. The {\it Krasnoselski iteration} \cite{Kra55, Sch57} $(x_n)$ starting at $x \in C$ is defined as in the case of normed spaces by
\[x_0 : = x, \quad x_{n+1} := (1 - \lambda)x_n + \lambda Tx_n,\]
where $\lambda \in (0,1)$. Note that the Picard iteration $(T_\lambda^n x )$ of the averaged mapping $T_\lambda : C \to C$, $T_\lambda(x) = (1-\lambda)x + \lambda Tx$ is in fact the Krasnoselski iteration starting at $x \in C$ and 
\[d\left(T_\lambda^n x, T_\lambda^{n+1} x\right) = \lambda d\left(T_\lambda^n x, T(T_\lambda^{n} x\right)) = \lambda d(x_n, T x_n).\]
This means that asymptotic regularity of the mapping $T_\lambda$ is equivalent to the fact that for each starting point $x \in C$, $\lim_{n \to \infty}d(x_n, T x_n) = 0$. Hence, the study of a rate of asymptotic regularity for the averaged mapping $T_\lambda$ reduces to computing a rate of asymptotic regularity for the Krasnoselski iteration of the nonexpansive mapping $T$. In the case of normed spaces, a uniform and quadratic rate was established by  Baillon and Bruck in \cite{BaiBru96} using a difficult computer aided proof. Afterwards, for the case $\lambda = 1/2$ a simple proof was given in \cite{Bru96}. Rates of asymptotic regularity for the more general Krasnoselski-Mann iteration were computed using methods of proof mining in the setting of normed spaces by Kohlenbach \cite{Koh01} and in the setting of hyperbolic spaces by Kohlenbach and Leu\c stean \cite{KohLeu03} by analyzing a result of Borwein, Reich and Shafrir \cite{BorReiSha92}. Although the main result in \cite{KohLeu03} was proved in hyperbolic spaces in the sense of \cite{ReiSha90}, the proof goes through for the setting of geodesic spaces with convex metric. We include below a rate of asymptotic regularity for averaged mappings in geodesic spaces with convex metric which is an immediate consequence of \cite[Corollary 3.18]{KohLeu03}. This rate only depends on $\varepsilon$, on an upper bound $b$ on the diameter of the set $C$ and on $\lambda$.

\begin{theorem} 
Let $(X,d)$ be a geodesic space with convex metric, $C \subseteq X$ nonempty convex and bounded with diameter $d_C \le b$. Suppose $T_\lambda : C \to C$ is averaged. Then,
\[\forall x \in C, \forall \varepsilon > 0, \forall n \ge \Phi(\varepsilon, b, \lambda), \quad d\left(T_\lambda^n x, T_\lambda^{n+1} x\right) \le \varepsilon,\]
where
\[\Phi(\varepsilon, b, \lambda) := KM\left\lceil 2 b e^{K(M+1)}\right\rceil,\]
with 
\[ K = \max\left\{\left\lceil \frac{1}{\lambda}\right\rceil, \left\lceil \frac{1}{1-\lambda}\right\rceil \right\}, \quad M = \left\lceil\frac{\lambda(1 + 2b)}{\varepsilon}\right\rceil.\]
\end{theorem}

The above rate is exponential in $1/\varepsilon$. Rates of asymptotic regularity for the Krasnoselski-Mann iteration were further studied in particular hyperbolic spaces (namely in $UCW$-hyperbolic spaces) by Leu\c stean \cite{Leu07}. Note that instead of $UCW$-hyperbolic spaces one can also consider the more general context of uniformly convex geodesic spaces with convex metric that admit a monotone modulus of uniform convexity since the proof carries over to this setting with no changes. In particular, the results proved in \cite{Leu07} guarantee a quadratic rate of asymptotic regularity (in $1/\varepsilon$) for CAT$(0)$ spaces which again only depends on $\varepsilon$, on an upper bound $b$ on the diameter of the set $C$ and on $\lambda$.

\subsection{Firmly nonexpansive mappings} \label{sect-effective}
In this section we mainly focus on providing a rate of asymptotic regularity for firmly nonexpansive mappings in the setting of geodesic spaces with convex metric. 

Let $(X,d)$ be a $W$-hyperbolic space, $C \subseteq X$ nonempty and suppose $T:C \to C$ is $\lambda$-firmly nonexpansive. It was proved in \cite{AriLeuLop12} that for every $x \in C$,
\[\lim_{n \to \infty} d\left(T^n x, T^{n+1} x\right) = r_C(T).\]
In fact, this holds in the context of geodesic spaces with convex metric. This property is the counterpart in the geodesic setting of a very important result given by Reich and Shafrir \cite{ReiSha87} in Banach spaces. Using methods of proof mining, we give in the sequel a quantitative version of this result. 
\begin{theorem} \label{th-fn-rate}
Let $(X,d)$ be a geodesic space with convex metric, $C \subseteq X$ nonempty and $T : C \to C$ $\lambda$-firmly nonexpansive.
Let $x,y \in C$ and $b \ge \max\{d(x,y), d(x,Tx)\}$. Then,
\[\forall \varepsilon > 0, \forall n \ge \Phi(\varepsilon, b, \lambda), \quad d(T^n x, T^{n+1} x) \le d(y,Ty) + \varepsilon,\]
where
\begin{equation}\label{th_fn-rate-exp}
\Phi(\varepsilon, b, \lambda) := M\left\lceil \frac{2 b (1 + e^{KM})}{\varepsilon}\right\rceil,
\end{equation}
with 
\[ K = \left\lceil \frac{1}{\lambda}\right\rceil \quad \text{and} \quad M = \left\lceil\frac{4b}{\varepsilon}\right\rceil.\]
\end{theorem}

Although we considered the setting of geodesic spaces with convex metric, the result holds in metric spaces for which there exists a convexity mapping $W$ satisfying (W1)-(W3). Therefore, the above result holds in particular in every normed space and so our result is new even in the setting of normed spaces. Before proving the theorem let us give some consequences thereof. We remark first that we obtain the result proved in \cite{AriLeuLop12}.

\begin{corollary}
Let $(X,d)$ be a geodesic space with convex metric, $C \subseteq X$ nonempty and $T : C \to C$ $\lambda$-firmly nonexpansive. Then, for each $x \in C$, 
\[\lim_{n \to \infty} d\left(T^n x, T^{n+1} x\right) = r_C(T).\]
\end{corollary}
 
If the set $C$ is bounded, we can obtain a uniform rate of asymptotic regularity which only depends on $\varepsilon$, on an upper bound $b$ on the diameter of the set $C$ and on $\lambda$. This rate is exponential in $1/\varepsilon$.
\begin{corollary}
Let $(X,d)$ be a geodesic space with convex metric, $C \subseteq X$ a nonempty and bounded set with diameter $d_C \le b$. Suppose $T : C \to C$ is $\lambda$-firmly nonexpansive. Then,
\[\forall x \in C, \forall \varepsilon > 0, \forall n \ge \Phi(\varepsilon, b, \lambda), \quad d\left(T^n x, T^{n+1} x\right) \le \varepsilon,\]
where $\Phi(\varepsilon, b, \lambda)$ is given by (\ref{th_fn-rate-exp}). 
\end{corollary}
\begin{proof} The result follows by Theorem \ref{th-fn-rate} using the fact that $r_C(T) = 0$ when $C$ is bounded (see \cite[Corollaries 5.2, 5.3]{AriLeuLop12}).
\end{proof}

We remark that in the case of normed spaces, no effective bounds on the asymptotic regularity have been computed so far. Another rate of asymptotic regularity for $\lambda$-firmly nonexpansive mappings was obtained in the more particular setting of $UCW$-hyperbolic spaces in \cite{AriLeuLop12}. For the special case of CAT$(0)$ spaces, this result provides a much better rate since it is quadratic in $1/\varepsilon$, while in $L_p$ spaces with $1 < p < \infty$ it is polynomial in $1/\varepsilon$. Therefore, our result is not optimal for these classes of geodesic spaces. However, it is the only bound computed so far that holds in the wide setting of geodesic spaces with convex metric.

We prove next an inequality for mappings that satisfy condition (\ref{def_firmly-non}) which is needed in the proof of the main result of this section.
\begin{lemma}\label{lemma-fn}
Let $(X,d)$ be a geodesic space with convex metric, $C \subseteq X$ nonempty and suppose $T:C \to C$ satisfies (\ref{def_firmly-non}). Then, for each $x \in C$ and $i,n \ge 1$,
\begin{align*}
n d(T^i x, T^ {i+1} x) \le d(T^i x, T^{i+n} x) + n\left( 1 + \left(\frac{1 + \lambda}{\lambda}\right)^n\right) \left(d(T^i x, T^{i+1} x) - d(T^{i+n} x, T^{i+n+1} x) \right).
\end{align*} 
\end{lemma}
\begin{proof}
We prove first by induction that for each $i,n \ge 1$,
\begin{align}\label{lemma-fn-eq1}
 d(T^i x, T^{i+n} x) \ge n d(T^{i+n}x, T^{i+n+1} x) - n\left(\frac{1 + \lambda}{\lambda}\right)^n \left(d(T^i x, T^{i+1} x) - d(T^{i+n} x, T^{i+n+1} x) \right).
\end{align} 
For $n =1$, the above inequality obviously holds for each $i \ge 1$. Suppose (\ref{lemma-fn-eq1}) holds for each $i \ge 1$. We prove that for $i \ge 1$ we have
\begin{align*}
d(T^i x, T^{i+n + 1} x) &\ge (n+1) d(T^{i+n+1}x, T^{i+n+2} x)\\
& \ \ \ \ - (n+1)\left(\frac{1 + \lambda}{\lambda}\right)^{n+1} \left(d(T^i x, T^{i+1} x) - d(T^{i+n+1} x, T^{i+n+2} x) \right).
\end{align*}
Note that \cite[Lemma 5.6]{AriLeuLop12} also holds in geodesic spaces with convex metric. Thus,
\begin{align*}
d(T^i x, T^{i+n+1} x) & \ge \frac{1+\lambda}{\lambda} d(T^{i+1} x, T^{i+n+1} x) - \frac{1-\lambda}{\lambda} d(T^{i+n} x, T^i x) - d(T^{i+1} x, T^{i+n} x)\\
& \ge \frac{1+\lambda}{\lambda} n d(T^{i+n +1}x, T^{i+n+2} x) - n\left(\frac{1 + \lambda}{\lambda}\right)^{n+1} \left(d(T^{i+1} x, T^{i+2} x) - d(T^{i+n+1} x, T^{i+n+2} x) \right)\\
& \ \ \ \ - \frac{1-\lambda}{\lambda} n d(T^i x, T^{i+1} x) - (n-1) d(T^i x, T^{i+1} x)\\
& \ge (n+1) d(T^{i+n +1}x, T^{i+n+2} x) + \left(\frac{n}{\lambda} -1\right) d(T^{i+n +1}x, T^{i+n+2} x) - \left(\frac{n}{\lambda} -1\right) d(T^i x, T^{i+1} x)\\
&\ \ \ \  -  n\left(\frac{1 + \lambda}{\lambda}\right)^{n+1} \left(d(T^{i} x, T^{i+1} x) - d(T^{i+n+1} x, T^{i+n+2} x) \right)\\
& = (n+1) d(T^{i+n +1}x, T^{i+n+2} x)\\
& \ \ \ \ - \left(\frac{n}{\lambda} -1 + n\left(\frac{1 + \lambda}{\lambda}\right)^{n+1}\right) \left(d(T^{i} x, T^{i+1} x) - d(T^{i+n+1} x, T^{i+n+2} x) \right)\\
& \ge (n+1) d(T^{i+n+1}x, T^{i+n+2} x)\\
& \ \ \ \  - (n+1)\left(\frac{1 + \lambda}{\lambda}\right)^{n+1} \left(d(T^i x, T^{i+1} x) - d(T^{i+n+1} x, T^{i+n+2} x) \right).
\end{align*} 
Hence, $(\ref{lemma-fn-eq1})$ is proved. From $(\ref{lemma-fn-eq1})$ we immediately obtain the conclusion.
\end{proof}

\subsection{Proof of Theorem \ref{th-fn-rate}}

The methods that we employ are inspired by the ones used in \cite{Koh01} and \cite{KohLeu03} for computing rates of asymptotic regularity for the Krasnoselski-Mann iteration of nonexpansive and directionally nonexpansive mappings.

Let $x,y \in C$ and $\varepsilon > 0$. We may suppose that $x$ is not a fixed point of $T$ (otherwise the conclusion is obvious). Define $\gamma := d(y, Ty)$ and let $\displaystyle \delta := \frac{\varepsilon}{2(1 + e^{KM})}$, $\displaystyle j := \left\lceil\frac{d(x,Tx)}{\delta} \right\rceil$. Then there exists $1 \le k \le j$ such that
\begin{align} \label{th-ef-fn-eq1}
D_k := d\left(T^{kM} x,T^{kM +1} x\right) - d\left(T^{(k+1)M} x, T^{(k+1)M +1} x\right) \le \delta.
\end{align}
Indeed, supposing on the contrary that for each $k = 1, \ldots, j$, $D_k > \delta$ we obtain that $\displaystyle\sum_{k=1}^j D_k > \delta j$ which yields that
\[d\left(T^{M} x,T^{M +1} x\right) > d(x, Tx),\]
a contradiction.\\
Let $1 \le k \le j$ satisfy (\ref{th-ef-fn-eq1}). Apply Lemma \ref{lemma-fn} with $i := kM$ and $n := M$ to get that
\begin{align*}
M d\left(T^{kM} x, T^{kM +1} x\right) & \le  d\left(T^{kM} x, T^{(k+1)M} x\right)  \\
& \ \ \ \ + M\left( 1 + \left(\frac{1 + \lambda}{\lambda}\right)^M\right) \left(d(T^{kM} x, T^{kM+1} x) - d(T^{(k+1)M} x, T^{(k+1)M +1} x) \right).
\end{align*}
Note that
\[\left(\frac{1 + \lambda}{\lambda}\right)^M = e^{M \ln \left(1 + \frac{1}{\lambda}\right)} \le e^\frac{M}{\lambda} \le e^{KM}.\]
Using (\ref{th-ef-fn-eq1}) we obtain that
\begin{align*}
d\left(T^{kM} x, T^{kM +1} x\right) & \le \frac{1}{M}\left(d\left(T^{kM} x, T^{kM} y \right) + d\left(T^{kM} y, T^{(k+1)M} y \right) + d\left(T^{(k+1)M} y, T^{(k+1)M} x \right)\right)\\
& \ \ \ \ + \left(1 + e^{KM}\right)\delta\\
& \le \frac{1}{M}\left(2d(x,y) + M\gamma\right) + \left(1 + e^{KM}\right)\delta \le \gamma + \frac{2b}{M} +
\frac{\varepsilon}{2} = \gamma + \varepsilon.
\end{align*}
Since
\[k \le \left\lceil\frac{d(x,Tx)}{\delta} \right\rceil \le \left\lceil\frac{b}{\delta} \right\rceil = \left\lceil \frac{2 b (1 + e^{KM})}{\varepsilon}\right\rceil,\]
it follows that $kM \le \Phi$. Thus, for $n \ge \Phi$, 
\[d\left(T^n x, T^{n+1} x\right) \le d\left(T^{kM} x, T^{kM +1} x\right) \le \gamma + \varepsilon,\]
which ends the proof.

Note that in this result we may suppose that $T$ is not nonexpansive and solely satisfies condition (\ref{def_firmly-non}) if $b$ is chosen such that $b \ge d(x,Tx)$ and $b \ge d(T^n x,T^n y)$ for each $n  \in \mathbb{N}$.

\section{Effective rates of asymptotic regularity in solving the convex feasibility problem}\label{sect_conv-feas-pr}
In this section we focus on giving effective rates of asymptotic regularity for sequences obtained by two algorithms that are employed in the study of the convex feasibility problem in a nonlinear setting. In Hilbert spaces, the convex feasibility problem consists in finding a point in the intersection of finitely many nonempty closed and convex sets, assuming that there is such a point. There exist several algorithms which use metric projections on the sets in order to build sequences of points that converge to a point in this intersection. Bauschke and Borwein studied in \cite{BauBor96} a very general method for solving the convex feasibility problem which unifies several important algorithms. We consider next two such algorithmic schemes in a nonlinear setting and use methods of proof mining to give effective results on the asymptotic regularity of the sequences obtained. 

The first method we focus on is the alternating projection method for two sets developed by von Neumann \cite{vNeu33} which is probably the most well-known and straightforward algorithm (for an elementary geometric proof of von Neumann's theorem see \cite{KopRei04}). This method can also be used in a nonlinear setting \cite{Rei93, BacSeaSim12}. Moreover, if one assumes some regularity conditions on the sets, one obtains linear convergence of this method \cite{BauBor93, BacSeaSim12}. Here we compute a rate of asymptotic regularity for the sequence generated by the alternating projection algorithm when no additional conditions are considered on the sets. Note that in this case it is known that one can only obtain weak convergence of this sequence \cite{Hun04, MatRei03}.

Let us first give the precise definition of the metric projection. Let $X$ be a metric space and $Y \subseteq X$. We define the {\it distance of a point} $z \in X$ to $Y$ by $\displaystyle \mbox{dist}(z,Y) = \inf_{y \in Y}d(z,y)$. The {\it metric projection} $P_Y$ onto $Y$ is the mapping
\[P_Y(z)=\{ y \in Y : d(z,y)=\mbox{dist}(z,Y)\} \quad \mbox{for every } z\in X.\]
In any CAT(0) space the metric projection onto convex and complete subsets is a single-valued and firmly nonexpansive mapping \cite[Proposition 3.1]{AriLeuLop12}.

Let $(X,d)$ be a complete CAT$(0)$ space. The method of alternating projections for two sets can be described in the following way: let $A,B \subseteq X$ be nonempty closed and convex. Given $x_0 \in X$, define, for $n \ge 1$, 
\begin{equation} \label{AP_eq}
x_{2n-1} := P_A(x_{2n-2}) \quad \text{and} \quad x_{2n} := P_B(x_{2n-1}).
\end{equation}

We start with the following simple property of firmly nonexpansive mappings that have fixed points.

\begin{lemma} \label{AP_lemma}
Suppose $C$ is a nonempty subset of $X$ and $T:C \to C$ is firmly nonexpansive with a nonempty set of fixed points $\emph{Fix}(T)$. Let $x \in C$ and $p \in \emph{Fix}(T)$. Then,
\[d(x,Tx)^2 \le d(x,p)^2 - d(Tx, p)^2.\]
\end{lemma}
\begin{proof}
For each $\lambda \in (0,1)$,
\[d(Tx, p)^2 = d(Tx,Tp)^2 \le d\left((1-\lambda)x + \lambda Tx, p\right)^2 \le (1-\lambda)d(x,p)^2 + \lambda d(Tx,p)^2 - \lambda(1-\lambda)d(x,Tx)^2.\]
Thus,
\[d(Tx, p)^2 \le d(x,p)^2 - \lambda d(x,Tx)^2.\]
Letting $\lambda \nearrow 1$ we obtain the conclusion.
\end{proof}

The next result provides a rate of asymptotic regularity of the sequence $(x_n)$ generated by (\ref{AP_eq}). 

\begin{theorem}
Let $(X,d)$ be a complete CAT$(0)$ space and $A,B \subseteq X$ be nonempty closed and convex with $A \cap B \ne \emptyset$. Consider $x_0 \in X$ and $b > 0$ such that there exists $p \in A \cap B$ with $d(x_0,p) \le b$. Define the sequence $(x_n)$ by (\ref{AP_eq}). Then,
\[\forall \varepsilon > 0, \forall n \ge \Phi(\varepsilon, b), \quad d\left(x_n, x_{n+1} \right) \le \varepsilon,\]
where
\[
 \Phi(\varepsilon, b):=\begin{cases}\displaystyle
\left[\frac{b^2}{\varepsilon^2}\right] & \text 
{for  }\varepsilon<2b, \\
 0 & \text{otherwise}.
\end{cases}
\]
\end{theorem}
\begin{proof}
Let $p \in A \cap B = \text{Fix}(P_A) \cap \text{Fix}(P_B)$ with $d(x_0,p) \le b$. By \cite[Lemma 3.4]{BacSeaSim12} we know that $d(x_{n+1}, p) \le d(x_n, p)$ for all $n \ge 0$ and so,
\[d(x_n,x_{n+1}) \le d(x_n, p) + d(x_{n+1},p) \le 2b.\]
Hence, the case $\varepsilon \ge 2b$ is clear. Suppose $\varepsilon < 2b$ and define
\[N := \left[\frac{b^2}{\varepsilon^2}\right].\]
Since the metric projection onto closed and convex subsets is firmly nonexpansive, we can apply Lemma \ref{AP_lemma} to conclude that
\begin{equation}\label{AP-th-eq}
d(x_n, x_{n+1})^2 \le d(x_n, p)^2 - d(x_{n+1},p)^2.
\end{equation}
Indeed, without loss of generality we may assume that $x_n \in A$. Then $x_{n+1} = P_B(x_n)$, so
\[d(x_n, x_{n+1})^2 = d(x_n, P_B(x_n))^2 \le  d(x_n, p)^2 - d(P_B(x_n),p)^2 =  d(x_n, p)^2 - d(x_{n+1},p)^2.\]
Suppose that for each $n = 0, \ldots, N$, $d(x_n,x_{n+1}) > \varepsilon$. Since, by (\ref{AP-th-eq}),
\[\sum_{i=0}^N d(x_i, x_{i+1})^2 \le d(x_0,p)^2 - d(x_{N+1},p)^2 \le b^2,\]
it follows that $(N+1)\varepsilon^2 < b^2$, which is a contradiction. Thus, there exists $n \le N$ such that $d(x_n,x_{n+1}) \le \varepsilon$. Note that for each $i \ge 0$, $d(x_{i+1}, x_{i+2}) \le d(x_i, x_{i+1})$. To see this, fix $i \ge 0$. Again without loss of generality we can suppose that $x_i \in A$. Then, since $x_{i+2} = P_A(x_{i+1})$, we have that
\[d(x_{i+1}, x_{i+2}) = d(x_{i+1}, P_A(x_{i+1})) = \text{dist}(x_{i+1}, A) \le d(x_{i+1}, x_i).\]
Therefore, the sequence $(d(x_i, x_{i+1}))$ is non-increasing and we are done.
\end{proof}

The second algorithmic scheme that we deal with extends a method studied in Hilbert spaces by Crombez \cite{Cro91} and later generalized by Takahashi and Tamura \cite{TakTam96} to the setting of uniformly convex Banach spaces. This method focuses on the Picard iterates of a mapping defined in terms of weighted averages of nonexpansive retractions. We give in the sequel the precise definition in a nonlinear setting.

Let $(X,d)$ be a Busemann convex geodesic space and $C \subseteq X$ nonempty and convex. Let $r \in \mathbb{N}$, $r \ge 2$ and for every $i = 1, \ldots, r$, suppose that $T_i : C \to C$ is nonexpansive. Let $(\alpha_i)_{i=1}^r \subseteq (0,1)$ be a weight, that is, $\displaystyle \sum_{i = 1}^r \alpha_i = 1$. Define $\displaystyle a_j = \sum_{i = j}^r \alpha_i$, $j = 1, \ldots, r$. Define the mapping $S_0 : C \to C$, $S_0 := T_r$.\\ 
If $r \ge 3$, for $i = 1, \ldots, r-2$, consider
\[S_i : C \to C, \quad S_i := \frac{\alpha_{r-i}}{a_{r-i}} T_{r-i} + \frac{a_{r-i+1}}{a_{r-i}} S_{i-1}.\]
Note that $1-a_2 = \alpha_1$ and $S_i$ is nonexpansive for each $i = 0, \ldots, r-2$. 
We focus next on the nonexpansive mapping $T : C \to C$,
\begin{equation}\label{def-T-cv-comb}
T := \left(1 - a_2\right) T_1 + a_2 S_{r-2}.
\end{equation}
Remark that in the setting of normed spaces, the mapping $T$ is defined by $\displaystyle T:= \sum_{i=1}^r \alpha_i T_i$. Iterates of such mappings $T$ in Banach spaces are studied in \cite{Rei83}.

\begin{lemma} \label{lemma-image}
For $i = 1, \ldots, r$, let $C_i \subseteq C$ with $\displaystyle\bigcap_{i = 1}^r C_i \ne \emptyset$ and $P_i : C \to C_i$ be nonexpansive retractions. Define $T_i : C \to C$, $T_i = (1-\lambda_i) I + \lambda_i P_i$, where $\lambda_i \in (0,1)$. Then, $\emph{Fix}(T) = \displaystyle\bigcap_{i = 1}^r C_i$.
\end{lemma}
\begin{proof}
Since $\text{Fix}(T_i) = \text{Fix}(P_i) = C_i$, it is easy to see that $\displaystyle\bigcap_{i = 1}^r C_i \subseteq \text{Fix}(T)$. Pick $x \in \text{Fix}(T)$ and $y \in \displaystyle\bigcap_{i = 1}^r C_i$. We claim that 
\[ a_{2} d\left(S_{r-2} x, y \right) \le \sum_{i=2}^r \alpha_i d\left(T_i x,y\right).\]
If $r=2$ this obviously holds. Suppose $r \ge 3$. Then, for $i = 1, \ldots, r-2$, we have that
\[ a_{r-i} d\left(S_i x, y \right) \le \alpha_{r-i} d\left(T_{r-i} x,y\right) + a_{r-i+1} d\left(S_{i-1} x,y\right),\]
from where our claim follows.\\
Thus,
\begin{align*}
d(x,y) &= d(Tx,y) = d\left((1 - a_2) T_1 x + a_2 S_{r-2} x, y\right) \le (1-a_2) d\left(T_1 x, y\right) + a_2 d\left(S_{r-2} x, y\right)\\ 
& \le \sum_{i=1}^r \alpha_i d\left(T_i x,y\right) = \sum_{i=1}^r \alpha_i d\left((1-\lambda_i) x + \lambda_i P_i x,y\right)\\
& \le \sum_{i=1}^r \alpha_i \left((1-\lambda_i) d(x,y) + \lambda_i d(P_i x,y)\right) \le d(x,y).
\end{align*}
Hence, for each $i = 1, \ldots, r$,
\[d(P_i x, y) = d\left((1-\lambda_i)x + P_i x, y\right) = d(x,y).\]
But since $X$ is strictly convex it follows that $P_i x = x$ for each $i = 1, \ldots, r$ and so $x \in \displaystyle \bigcap_{i=1}^r C_i$.
\end{proof}

The following result provides a rate of asymptotic regularity for the mapping $T$. Note that this result also holds in particular in uniformly convex Banach spaces.

\begin{theorem} \label{th-image-effective}
Let $X$ be a $UCW$-hyperbolic space and $C \subseteq X$ nonempty and convex. For $i = 1, \ldots, r$, let $C_i \subseteq C$ with $\displaystyle\bigcap_{i = 1}^r C_i \ne \emptyset$ and $P_i : C \to C_i$ be nonexpansive retractions. Define $T_i : C \to C$ by $T_i = (1-\lambda_i) I + \lambda_i P_i$, where $\lambda_i \in (0,1)$, $i = 1, \ldots, r$. Let $x \in C$ and $b > 0$ such that there exists $p \in \displaystyle\bigcap_{i = 1}^r C_i$ with $d(x,p) \le b$. Denote $\alpha = (\alpha_1, \ldots, \alpha_r)$, $\lambda=(\lambda_1, \ldots, \lambda_r)$ and, by $\delta$, a monotone modulus of uniform convexity. Then,
\[\forall \varepsilon > 0, \forall n \ge \Phi(\varepsilon, b, \delta, \lambda, \alpha), \quad d\left(T^n x, T^{n+1} x\right) \le \varepsilon,\]
where
\[
 \Phi(\varepsilon, b, \delta, \lambda, \alpha):=\begin{cases}\displaystyle
\left[\frac{b}{\displaystyle \varepsilon  K \delta\left(b,\varepsilon/b\right)}\right] & \text 
{for  }\varepsilon<2b, \\
 0 & \text{otherwise},
\end{cases}
\]
with $K = \min\{\alpha_i \lambda_i (1-\lambda_i) : i = 1, \ldots, r\}$.
\end{theorem}
\begin{proof}
Let $p \in \displaystyle\bigcap_{i = 1}^r C_i$ with $d(x,p) \le b$. By Lemma \ref{lemma-image}, $p \in \text{Fix}(T)$. The case $\varepsilon \ge 2b$ is obvious. Assume $\varepsilon < 2b$ and define
\[N := \left[\frac{b}{\displaystyle \varepsilon  K \delta\left(b,\varepsilon/b\right)}\right].\]
Suppose that for each $n = 0, \ldots, N$, $d\left(T^n x, T^{n+1} x\right) > \varepsilon$. Fix such an $n$. Clearly, $d\left(T^n x, p\right) \le b$. Since
\begin{align*}
\varepsilon < d\left(T^n x, T^{n+1} x\right) = d\left(T^n x, \left(1 - a_2\right) T_1(T^n x) + a_2 S_{r-2} (T^n x)\right) \le \sum_{i=1}^r \alpha_i d\left(T_i(T^n x), T^n x\right),
\end{align*}
it follows that there exists $1 \le i = i(n) \le r$ with $d\left(T_{i}(T^n x), T^n x\right) > \varepsilon$ and so $d\left(P_{i}(T^n x), T^n x\right) > \varepsilon$.\\
Because $d\left(P_i(T^{n} x),p\right) \le d\left(T^n x, p\right)$ we obtain that (see for instance \cite[Lemma 7]{Leu07})
\begin{align*}
d(T_i(T^{n} x), p)  = d\left((1-\lambda_i)T^{n} x +\lambda_i P_i(T^{n} x), p\right) \le \left(1 - 2\lambda_i(1 - \lambda_i)\delta\left(b, \varepsilon/b\right)\right)d\left(T^n x, p\right).
\end{align*}
Note that 
\[\varepsilon < d\left(T^n x, T^{n+1} x\right) \le d\left(T^n x, p\right) + d\left(T^{n+1} x, p\right) \le 2d\left(T^n x, p\right).\]
Thus, we obtain that
\[d(T_i(T^{n} x), p) \le d\left(T^n x, p\right) - \varepsilon\lambda_i(1 - \lambda_i)\delta\left(b,\varepsilon/b\right).\]
However,
\begin{align*}
d(T^{n+1} x, p) &= d\left((1 - a_2) T_1(T^n x) + a_2 S_{r-2} (T^n x), p\right) \le \sum_{j=1}^r \alpha_j d\left(T_j(T^n x), p\right)\\
& \le \alpha_i d\left(T_i(T^n x), p\right)  + (1-\alpha_i) d(T^n x,p).
\end{align*}
This means that
\begin{align*}
d(T^{n+1} x, p) \le d\left(T^n x,p\right) - \varepsilon \alpha_i \lambda_i(1 - \lambda_i)\delta\left(b,\varepsilon/b\right) \le d\left(T^n x,p\right) - \varepsilon K\delta\left(b,\varepsilon/b\right).
\end{align*}
Hence,
\begin{align*}
d(T^{N+1} x, p) \le d(x,p) - (N+1)\varepsilon K\delta\left(b,\varepsilon/b\right) < 0,
\end{align*}
a contradiction. Thus, there exists $n \le N$ such that $d\left(T^n x, T^{n+1} x\right) \le \varepsilon$. Since the sequence $\left(d(T^n x, T^{n+1} x)\right)$ is non-increasing the conclusion follows.
\end{proof}

\begin{remark}
If, in addition, we have that $\delta(r, \varepsilon) \ge \varepsilon \cdot \tilde \delta(r, \varepsilon)$ such that $\tilde \delta$ increases with respect to $\varepsilon$ for $r$ fixed, then $\Phi(\varepsilon, b, \delta, \lambda, \alpha)$ can be replaced for $\varepsilon < 2b$ with
\[ \tilde \Phi(\varepsilon, b, \delta, \lambda, \alpha) = \left[\frac{b}{2 \varepsilon  K\tilde \delta\left(b,\varepsilon/b\right)}\right].\]
\end{remark}
\begin{proof}
Define $r_n : = d\left(T^n x, p\right)$ and
\[N:= \left[\frac{b}{2 \varepsilon  K\tilde \delta\left(b,\varepsilon/b\right)}\right].\]
Like before, we have that
\begin{align*}
d(T_i(T^{n} x), p)  &\le \left(1 - 2\lambda_i(1 - \lambda_i)\delta\left(r_n, \varepsilon/r_n\right)\right)r_n \le r_n  - 2r_n \lambda_i(1 - \lambda_i)\delta\left(b, \varepsilon/r_n\right)\\
& = r_n -2\varepsilon \lambda_i(1 - \lambda_i)\tilde \delta\left(b, \varepsilon/r_n\right) \le r_n -2\varepsilon \lambda_i(1 - \lambda_i)\tilde \delta\left(b, \varepsilon/b\right).
\end{align*}
The conclusion follows using the same reasoning as above.
\end{proof}

Since any CAT$(0)$ space is a $UCW$-hyperbolic space for which $\delta(r,\varepsilon) = \varepsilon^2 /8$ is a modulus of uniform convexity we obtain in this case a quadratic rate of asymptotic regularity in $1/\varepsilon$: for $\varepsilon < 2b$,
\[ \tilde \Phi(\varepsilon, b, \lambda, \alpha) = \left[\frac{4b^2}{\varepsilon^2  K}\right].\]

Let $(\Omega, \Sigma, \mu)$ be a measure space. Recall that for $1 < p < \infty$, $L_p(\mu)$ is uniformly convex and (see, for instance, \cite{Han56}), for $\varepsilon \in [0,2]$, the modulus of uniform convexity is bounded by
\[
\delta_{L_p}(\varepsilon) \ge \begin{cases}\displaystyle
\frac{p-1}{8}\varepsilon^2 & \text 
{for  }1<p\le2, \\
\displaystyle \frac{1}{p 2^p}\varepsilon^p & \text{for } 2<p < \infty.
\end{cases}
\]
Thus, for $\varepsilon < 2b$, we obtain the following expression for the rate of asymptotic regularity
\[
\tilde \Phi(\varepsilon, b, \lambda, \alpha) = \begin{cases}\displaystyle
\frac{4b^2}{\varepsilon^2K(p-1)} & \text 
{for  }1<p\le2, \\
\displaystyle \frac{p 2^{p-1}b^p}{\varepsilon^p K} & \text{for } 2<p < \infty.
\end{cases}
\]

\subsection{$\Delta$-convergence of the method}
In the sequel we briefly discuss the weak convergence of the above algorithm. More precisely, we generalize to a nonlinear setting weak convergence results proved by Crombez \cite{Cro91} in Hilbert spaces and Takahashi and Tamura \cite{TakTam96} in uniformly convex Banach spaces (see also \cite{Rei83}). First we need to recall a concept of weak convergence in metric spaces.

Let $(X,d)$ be a metric space and $(x_n)$ a bounded sequence in $X$. The {\it asymptotic radius} of $(x_n)$ is given by
\[r((x_n)) = \inf\left\{\limsup_{n \to \infty}d(x,x_n) : x \in X\right\},\]
and the {\it asymptotic center} of $(x_n)$ is the possibly empty set
\[A((x_n)) = \left\{x \in X : \limsup_{n \to \infty}d(x,x_n) = r((x_n))\right\}.\]
Note that an element of $A((x_n))$ will also be referred to as an asymptotic center.\\
In any complete uniformly convex metric space with a monotone modulus of uniform convexity, every bounded sequence has a unique asymptotic center \cite{EspFerPia10}. 

We give next a concept of convergence in metric spaces known as $\Delta$-convergence which was introduced by Lim in \cite{Lim76} and further studied by Kirk and Panyanak in \cite{KirPan08} where it was shown that in CAT$(0)$ spaces this concept is somewhat similar to that of weak convergence in Banach spaces. Another concept involves projections on geodesic segments and was given by Jost \cite{Jos94}. This notion proved to be equivalent to $\Delta$-convergence in the framework of CAT$(0)$ spaces \cite{EspFer09}. A bounded sequence $(x_n)$ in a metric space $X$ is said to {\it $\Delta$-converge} to $x \in X$ if $x$ is the unique asymptotic center of $(u_n)$ for every subsequence $(u_n)$ of $(x_n)$.

Consider now the mapping $T$ as defined by (\ref{def-T-cv-comb}). We adapt next to the setting of CAT$(0)$ spaces a method studied by Crombez \cite{Cro91} in Hilbert spaces. 
\begin{corollary}
Let $X$ be a complete CAT$(0)$ space and for $i = 2, \ldots, r$ let $C_i \subseteq X$ be nonempty closed and convex sets with $\displaystyle \bigcap_{i=2}^r C_i\ne \emptyset$. Define $T_1 : X \to X$, $T_1:=I$ and for $i =2, \ldots, r$, $T_i : X \to X$ by $T_i := (1-\lambda_i)I +\lambda_i P_i$, where $\lambda_i \in (0,1)$ and $P_i$ is the metric projection of $X$ onto $C_i$. Then, for all $x \in X$, the Picard iterate $(T^n x)$ $\Delta$-converges to an element of $\displaystyle \bigcap_{i=2}^r C_i$.
\end{corollary}
\begin{proof}
Since $T$ is in this case averaged and nonexpansive, and, by Lemma \ref{lemma-image}, $\text{Fix}(T) \ne \emptyset$, we know that $T$ is asymptotically regular. Apply \cite[Proposition 6.3]{AriLeuLop12} to obtain the conclusion.
\end{proof}

More generally, we have the following extension of \cite[Theorem 3.3]{TakTam96} in the geodesic setting.
\begin{corollary}
Let $X$ be a complete $UCW$-hyperbolic space and $C \subseteq X$ nonempty closed and convex. For $i = 1, \ldots, r$, let $C_i \subseteq C$ with $\displaystyle\bigcap_{i = 1}^r C_i \ne \emptyset$ and $P_i : C \to C_i$ be nonexpansive retractions. Define $T_i : C \to C$ by $T_i = (1-\lambda_i) I + \lambda_i P_i$, where $\lambda_i \in (0,1)$, $i = 1, \ldots, r$. Then, for all $x \in X$, the Picard iterate $(T^n x)$ $\Delta$-converges to an element of $\displaystyle \bigcap_{i=1}^r C_i$.
\end{corollary}
\begin{proof}
Since by Theorem \ref{th-image-effective} it follows that $T$ is asymptotically regular, we only need to apply \cite[Proposition 6.3]{AriLeuLop12} to obtain the conclusion.
\end{proof}

\section*{Acknowledgements}
I would like to thank Lauren\c tiu Leu\c stean for fruitful discussions on this topic. This research was supported by a grant of the Romanian National Authority for Scientific Research, CNCS – UEFISCDI, project number PN-II-ID-PCE-2011-3-0383.

\end{document}